\documentclass[12pt]{amsart}
\usepackage[T1]{fontenc}
\usepackage{lmodern}
\usepackage{amsmath,amssymb,mathtools}
\usepackage[a4paper,left=28mm,right=28mm,top=27mm,bottom=29mm]{geometry}
\usepackage{microtype,xcolor,needspace}
\usepackage[colorlinks=true,linkcolor=blue!45!black,citecolor=blue!45!black,urlcolor=blue!45!black]{hyperref}
\numberwithin{equation}{section}
\newtheorem{theorem}{Theorem}[section]
\newtheorem{proposition}[theorem]{Proposition}
\newtheorem{lemma}[theorem]{Lemma}

\theoremstyle{definition}

\theoremstyle{remark}
\newtheorem{remark}[theorem]{Remark}
\newcommand{\R}{\mathbb R}
\newcommand{\C}{\mathbb C}
\newcommand{\Sph}{\mathbb S}
\newcommand{\dd}{\,dV_g}
\newcommand{\ii}{\mathrm i}
\newcommand{\eps}{\varepsilon}
\newcommand{\kap}{\kappa}
\newcommand{\tr}{\operatorname{tr}}

\newcommand{\Div}{\operatorname{div}}
\newcommand{\HS}{\mathrm{HS}}
\newcommand{\mf}{\mathsf F}
\newcommand{\ip}[2]{\langle #1,#2\rangle}
\newcommand{\norm}[1]{\lVert #1\rVert}
\newcommand{\abs}[1]{\lvert #1\rvert}
\newcommand{\cT}{\mathcal T}
\newcommand{\cG}{\mathcal G}
\newcommand{\cC}{\mathcal C}

\newcommand{\sphere}{\Sph^3}
\newcommand{\E}{\mathcal E}
\AtBeginDocument{%
  \setlength{\abovedisplayskip}{12pt plus 3pt minus 3pt}%
  \setlength{\belowdisplayskip}{12pt plus 3pt minus 3pt}%
  \setlength{\abovedisplayshortskip}{6pt plus 2pt}%
  \setlength{\belowdisplayshortskip}{9pt plus 2pt minus 2pt}%
}
\allowdisplaybreaks[1]
\hypersetup{pdftitle={Stable Abelian Yang--Mills--Higgs Fields on the Round Three-Sphere}}
\title[Stable abelian Higgs fields on the three-sphere]{Stable Abelian Yang--Mills--Higgs Fields\\on the Round Three-Sphere}
\date{September 18, 2026}
\begin{document}
	\author{Xishen Jin}
	\address{Xishen Jin\\ Department of Mathematics\\ Renmin University of China \\ Beijing\\ 100872\\ China\\} \email{jinxishen@ruc.edu.cn}
\begin{abstract}
We prove that every stable weak critical point of the self-dual abelian
Yang--Mills--Higgs energy on the round three-sphere is gauge equivalent to
the vacuum, for every positive scale parameter. This establishes the
three-dimensional statement proposed by Cheng. For the non-stable critical point, we give a quantitative upper bound for the lowest eigenvalue of
the Hessian on the gauge quotient.
\end{abstract}
\maketitle

\section{Introduction}

Let $L$ be a Hermitian line bundle over a closed Riemannian manifold
$(M,g)$. For a unitary connection $D$ and a section $u$, write
$F_D=\ii D^2$ for the real curvature form and consider
\begin{equation}\label{eq:energy-eps}
 \E_\eps(D,u)=\frac12\int_M
 \left(\eps^2\abs{F_D}^2+\abs{Du}^2
       +\frac{(1-\abs u^2)^2}{4\eps^2}\right)\dd,
 \qquad \eps>0.
\end{equation}
A critical point is stable if the second variation is nonnegative for
all simultaneous connection and section variations. We prove that, on
the round three-sphere, stability forces $F_D=Du=0$ and $\abs u=1$. This resolves the three-dimensional case proposed by Cheng
\cite[Remark~1.7(2)]{Cheng2020} and, in the abelian setting,
completes the sphere stability picture developed by
Han--Jin--Wen \cite{HanJinWen2023} by treating the remaining
dimension three.

The coefficients in~\eqref{eq:energy-eps} are those of the self-dual
abelian Higgs model. On an oriented surface, the Bogomolny decomposition
expresses the energy as squares and a topological term; vanishing of
the squares gives the vortex equations
\cite{Bogomolny1976,JaffeTaubes1980}. Taubes established planar existence
with prescribed vortex zeros \cite{Taubes1980}. On compact surfaces,
the degree and area enter the existence criterion, as developed by
Bradlow and Garc\'ia-Prada \cite{Bradlow1990,GarciaPrada1994}.
These results describe minimizers; the converse from stability to the
first-order equations is a separate question. 

The higher-dimensional energy is related to codimension-two minimal
submanifolds. Pigati--Stern \cite{PigatiStern2021} obtained stationary
integral varifold limits of bounded-energy critical points as
$\eps\to0$. Parise--Pigati--Stern proved the corresponding
$\Gamma$-convergence and a comparison of min--max values
\cite{PPSGamma2024}, and related the parabolic equations to Brakke flow
\cite{PPSFlow2024}. Conversely, De~Philippis--Pigati constructed
critical points concentrating along non-degenerate minimal submanifolds
\cite{DePhilippisPigati2024}. Second inner variations and limiting index
bounds were studied by Marx-Kuo \cite{MarxKuo2025}. These results motivate
stability questions at fixed $\eps$, where no limiting argument is
available. In Euclidean space, De~Philippis--Halavati--Pigati established
excess decay and rigidity under a near-minimal energy density
assumption \cite{DHPExcess2026}.

Sphere stability arguments often sum second variations over
geometric families of vector fields. This method appears in
the work of Simons and Lawson--Simons
\cite{Simons1968,LawsonSimons1973}, and in the Yang--Mills
results of Bourguignon--Lawson--Simons and Bourguignon--Lawson
\cite{BourguignonLawsonSimons1979,BourguignonLawson1981}.
In dimension three, Stern's argument implies that stable
Yang--Mills connections on the round three-sphere are flat
\cite{Stern2010}. For the coupled Yang--Mills--Higgs system,
however, the connection equation contains a Higgs current,
and this conclusion does not apply directly. Cheng proved that every stable weak critical point of the
self-dual abelian Yang--Mills--Higgs energy on the round
$\Sph^n$, $n\ge4$, is gauge equivalent to the vacuum
\cite[Theorem~1.6]{Cheng2020}. He proposed the same conclusion
in dimension three \cite[Remark~1.7(2)]{Cheng2020}.
His subsequent work relates stability to the vortex equations
on $\Sph^2$ and $\mathbb T^2$ \cite{Cheng2021}.
For the more general (non-abelian) Yang--Mills--Higgs systems considered by
Han--Jin--Wen \cite{HanJinWen2023}, stability on $\Sph^n$
forces the Higgs field to be parallel and of unit length
when $n\ge4$. The connection is then Yang--Mills, and its
curvature vanishes when $n\ge5$.

The difficulty in dimension three is visible in the conformal
trace calculation. Let $X_\alpha$ be the gradients of the
ambient coordinate functions on $\Sph^n$. The variations used in \cite{HanJinWen2023} take
the form
\[
 V_\alpha=(\iota_{X_\alpha}F_D,D_{X_\alpha}u),
\]
and satisfy
\[
 \sum_{\alpha=1}^{n+1}Q(V_\alpha)
 =2\eps^2(4-n)\norm{F_D}_2^2
   +(2-n)\norm{Du}_2^2
\]
where $Q(V_\alpha)$ is the second variation along $V_\alpha$. At $n=3$, the curvature term has a positive coefficient,
so this identity alone does not force $F_D$ and $Du$ to vanish.

A different family of variations was introduced by Badran
in his recent study of stable entire solutions of the abelian Higgs
equations \cite{Badran2026}. He obtained rigidity results
in $\R^4$ under quadratic energy growth and a vortex
classification in $\R^3$ under linear energy growth.
His construction uses the matrix absolute value of the
curvature together with multiplication of the section
variation by $\ii$.

We adopt this construction on the sphere. Let $B$ be the nonnegative
symmetric tensor given by the absolute value of the curvature operator,
as defined intrinsically in Section~\ref{sec:absolute}. Then we use the variation
\[
 W_\alpha
 =\bigl(-B(X_\alpha,\cdot),\ii D_{X_\alpha}u\bigr)
\]
as in \cite{Badran2026}. The underlying conformal gradient fields $X_\alpha$ are the same as
those in $V_\alpha$. 
A direct computation then gives
\[
 \sum_{\alpha=1}^{4}Q(W_\alpha)
 \le -2\eps^2\norm{F_D}_2^2-\norm{Du}_2^2,
\]
which resolves the three-dimensional stability problem.

\begin{remark}
While this manuscript was being completed, Badran, Guaraco, and
Halavati \cite{BGH2026} adapted the same absolute-curvature test
pairs to general oriented Riemannian three-manifolds. Using localized
variations, they obtained quantitative integral estimates for stable
fields and proved that bounded-energy sequences with
$\varepsilon\to0$ have limiting interfaces given by finite unions of
immersed geodesics.
\end{remark}

We now state the main result in the weak formulation used by
Cheng \cite{Cheng2020}. Throughout the paper,
$\sphere\subset\R^4$ denotes the unit round three-sphere.
Since $H^2(\sphere;\mathbb Z)=0$, every Hermitian line bundle
over $\sphere$ is trivial. Fixing a unitary trivialization,
we write
\[
 D=d-\ii A,\qquad F=F_D=dA.
\]
Gauge transformations act by
\begin{equation}\label{eq:gauge-intro}
 (A,u)\longmapsto(A+d\chi,e^{\ii\chi}u).
\end{equation}
Every smooth map $\sphere\to S^1$ admits a real lift because
$\sphere$ is simply connected. We call the gauge orbit
of $(0,1)$ the vacuum.

Following Cheng, we consider the configuration space
\begin{equation}\label{eq:weak-class-intro}
 \cC=\{(A,u): A\in W^{1,2}(T^*\sphere;\R),\quad
             u\in W^{1,2}(\sphere;\C)\cap L^\infty\}.
\end{equation}
A pair in $\cC$ is a weak critical point if its first variation
vanishes against every smooth pair, and is stable if its
second variation is nonnegative for every such pair.
As shown in Appendix~\ref{app:regularity}, every weak critical
point admits a smooth representative under a gauge
transformation $e^{\ii\chi}$ with $\chi\in W^{2,2}$,
and stability extends to $H^1=W^{1,2}$ variations.

\begin{theorem}\label{thm:main}
For every $\eps>0$, every stable weak critical point $(A,u)\in\cC$ of
$\E_\eps$ on the unit round three-sphere is gauge equivalent to $(0,1)$.
Equivalently, $F=0$, $Du=0$, and $\abs u=1$.
Conversely, every pair satisfying these identities is stable.
\end{theorem}

Put $\kap=\eps^2$. With the real inner product on complex quantities,
define
\begin{equation}\label{eq:weighted-intro}
\begin{aligned}
 \norm{(a,w)}_\kap^2&=\int_{\sphere}(\kap\abs a^2+\abs w^2)\dd,\\
 \cG(a,w)&=\kap d^*a+\ip{\ii u}{w}.
\end{aligned}
\end{equation}
The equation $\cG(a,w)=0$ defines the orthogonal complement of the
infinitesimal gauge directions $(d\chi,\ii u\chi)$.
Let $\mu_1$ be the infimum of $Q(Z)/\norm Z_\kap^2$ over nonzero
$H^1$ variations in this complement. In Section~\ref{sec:spectrum},
we justify its interpretation as the lowest eigenvalue.

\begin{theorem}\label{thm:gap}
Let $(A,u)\in\cC$ be a weak critical point. 

\begin{itemize}
    \item If
$2\eps^2\norm F_2^2+\norm{Du}_2^2>0$, then $\mu_1\le-1$.
\item The only nonvacuum critical point with
$2\eps^2\norm F_2^2+\norm{Du}_2^2=0$ is, up to gauge, $(0,0)$,
where $\mu_1=-1/(2\eps^2)$.
\end{itemize} 
 Thus every nonvacuum critical point satisfies
\[
 \mu_1\le-\min\left\{1,\frac1{2\eps^2}\right\}.
\]
\end{theorem}


\noindent\textbf{Organization Of The Paper:} In Section~\ref{sec:variation}, we record the field equations and
the Hessian. In Section~\ref{sec:absolute}, we establish the
properties of the absolute curvature tensor. In Section~\ref{sec:direct}, we compute the Hessian trace and prove the stability classification. In Section~\ref{sec:spectrum}, we prove the eigenvalue estimate.

\section{first and second variation of Energy}\label{sec:variation}

We work on the unit round three-sphere $\sphere$ and put $\kap=\eps^2$.
In a unitary trivialization, $D=d-\ii A$ and $F=dA$.
The energy is
\[
 E_\kap(A,u)=\frac12\int_{\sphere}
 \left(\kap\abs F^2+\abs{Du}^2
       +\frac{(1-\abs u^2)^2}{4\kap}\right)\dd,
 \qquad E_\kap=\E_{\sqrt\kap}.
\]
All inner products of complex quantities are real:
\begin{equation}\label{eq:complex-identities}
 \ip zw=\operatorname{Re}(z\overline w),
 \qquad \ip{\ii z}w=-\ip z{\ii w}.
\end{equation}
We use the induced real inner products on complex-valued tensors.
In particular, $\ip{Du}{\ii u}$ denotes the real one-form
$X\mapsto\ip{D_Xu}{\ii u}$; it will always be written in this form.
We write $\delta=d^*$ and use the nonnegative Laplacians
\[
\Delta= \Delta_H=d\delta+\delta d.
\]
The Levi--Civita connection is denoted by $\nabla$, and its coupling
with $D$ on complex-valued tensors by $\nabla^A$.
The musical isomorphism $a\mapsto a^\sharp$ is defined by the metric.
Unless otherwise indicated, all integrals and norms are on $\sphere$.

\subsection{The First Variation}

Along the affine path $(A_t,u_t)=(A+ta,u+tw)$,
\[
 F_{A_t}=F+t\,da,
\]
and
\[
 D_{A_t}u_t=Du+t(Dw-\ii au)-\ii t^2aw.
\]
The potential part in the energy is differentiated using
\[
 1-\abs{u+tw}^2
 =1-\abs u^2-2t\ip uw-t^2\abs w^2.
\]
Consequently,
\begin{align}
 \delta E_\kap(a,w)
 ={}&\int\left[\kap\ip F{da}+\ip{Du}{Dw-\ii au}
       -\frac{1-\abs u^2}{2\kap}\ip uw\right]\dd
       \notag\\
 ={}&\int\ip{\kap\delta F-\ip{Du}{\ii u}}a\dd
       +\int\ip{D^*Du-\frac{1-\abs u^2}{2\kap}u}w\dd.
 \label{eq:first-variation}
\end{align}
Thus the critical point equations are
\begin{equation}\label{eq:EL}
 \kap\delta F=\ip{Du}{\ii u},\qquad
 D^*Du=\frac{1-\abs u^2}{2\kap}u.
\end{equation}
We also recall the first Bianchi identity $dF=0$, which follows from $F=dA$.

\subsection{The Second Variation }

At a critical point, the second variation along $(a,w)$ is
\begin{equation}\label{eq:raw-hessian}
\begin{split}
 Q(a,w)={}&\int\left[
 \kap\abs{da}^2+\abs{Dw-\ii au}^2
 -2\ip{Du}{\ii aw}\right]\dd\\
 &+\frac1\kap\int\left[
 (\ip uw)^2-\frac{1-\abs u^2}{2}\abs w^2\right]\dd.
\end{split}
\end{equation}
Metric compatibility gives the identity of real one-forms
\[
 d\ip{\ii u}w=\ip{\ii Du}w+\ip{\ii u}{Dw}.
\]
It follows that
\begin{align*}
 -2\ip{Dw}{\ii au}-2\ip{Du}{\ii aw}
 &=-2\ip a{d\ip{\ii u}w}
       +4\ip{\ii D_{a^\sharp}u}w,\\
 \int\ip a{d\ip{\ii u}w}\dd
 &=\int\ip{\ii u}w\,\delta a\dd.
\end{align*}
where $a^\sharp$ denote the metric dual of a one-form $a$,
characterized by
\[
 g(a^\sharp,X)=a(X)
 \qquad\text{for every }X\in TM.
\]
Moreover,
\[
 (\ip uw)^2+(\ip{\ii u}w)^2=\abs u^2\abs w^2.
\]
Substituting these identities into~\eqref{eq:raw-hessian} gives
\begin{equation}\label{eq:hessian-integrated}
\begin{split}
 Q(a,w)={}&\int\left[
 \kap\abs{da}^2+\abs{Dw}^2+\abs u^2\abs a^2
 +4\ip{\ii D_{a^\sharp}u}w\right]\dd\\
 &+\int\left[
 -2\ip{\ii u}w\,\delta a
 +\frac{3\abs u^2-1}{2\kap}\abs w^2
 -\frac1\kap(\ip{\ii u}w)^2\right]\dd.
\end{split}
\end{equation}

\section{The absolute curvature tensor}\label{sec:absolute}

By Proposition~\ref{prop:regularity}, every weak critical point has a
smooth gauge representative, and stability is preserved.
We work with such a representative in this section and the next. In this section, we adapt to the round sphere the absolute-curvature
construction introduced by Badran in Euclidean space
\cite[Section~3.1]{Badran2026}. Whereas Badran defines the absolute
curvature as the absolute value of a skew-symmetric matrix in fixed
Euclidean coordinates, we regard the curvature as a skew-adjoint
bundle endomorphism and formulate the construction intrinsically.
The resulting tensor is therefore globally defined and independent
of the choice of local orthonormal frame.

The curvature defines a skew-adjoint endomorphism $\mf$ of the real
cotangent bundle by
\begin{equation}\label{eq:curvature-operator}
 (\mf\xi)(X)=\xi\bigl((\iota_XF)^\sharp\bigr),
 \qquad \xi\in T^*\sphere,\quad X\in T\sphere.
\end{equation}
We extend this endomorphism complex linearly. In an orthonormal coframe
its matrix is $(F(e_i,e_j))_{ij}$. In particular,
\begin{equation}\label{eq:two-form-norm}
 \mf^*=-\mf,\qquad \norm{\mf}_{\HS}^2=2\abs F^2.
\end{equation}
The factor $2$ comes from the convention
$\abs F^2=\sum_{i<j}F(e_i,e_j)^2$.

Following Badran's construction, define
\begin{equation}\label{eq:B-definition}
 B=\abs{\ii\mf}=\sqrt{-\mf^2}.
\end{equation}
Thus $B$ is the nonnegative self-adjoint square root of $-\mf^2$. Since $\mf^*=-\mf$, the endomorphism $-\mf^2=\mf^*\mf$
is self-adjoint and nonnegative, and so is its nonnegative
square root $B$.
We also regard it as a symmetric covariant two-tensor, with the
convention
\[
 B(X,Y)=\ip{B(X^\flat)}{Y^\flat}
\]
where $X^\flat$ is the metric-dual one form of $X$. 
The divergence of this symmetric tensor is the real one-form
\begin{equation}\label{eq:div-convention}
 (\Div B)(X)
 =\operatorname{tr}_g\bigl[(Y,Z)\mapsto(\nabla_YB)(Z,X)\bigr].
\end{equation}
This convention gives $\Div B=\sum_i(\nabla_{e_i}B)(e_i,\cdot)$
in any local orthonormal frame.

\begin{lemma}\label{lem:absolute-algebra}
If $F$ is smooth, then $B$ is locally Lipschitz. Pointwise,
\begin{equation}\label{eq:absolute-algebra}
 B^2=\mf^*\mf,\qquad
 \norm B_{\HS}^2=2\abs F^2,\qquad
 \tr B=2\abs F.
\end{equation}
Furthermore, for every complex-valued one-form $\xi$,
\begin{equation}\label{eq:absolute-positive}
 \ip{(B+\ii\mf)\xi}{\xi}\ge0.
\end{equation}
The covariant derivatives satisfy
\begin{equation}\label{eq:absolute-derivative-equality}
 \norm{\nabla B}_{\HS}^2=2\abs{\nabla F}^2
 \qquad\text{almost everywhere}.
\end{equation}
\end{lemma}

\begin{proof}
The identity $B^2=\mf^*\mf$ follows from skew-adjointness.
At a point where $F\ne0$, choose an oriented orthonormal coframe
with $F=r\,e^1\wedge e^2$, where $r=\abs F$.
Then
\begin{equation}\label{eq:absolute-normal}
 \mf=\begin{pmatrix}0&r&0\\-r&0&0\\0&0&0\end{pmatrix},
 \qquad B=\begin{pmatrix}r&0&0\\0&r&0\\0&0&0\end{pmatrix}.
\end{equation}
The norm and trace identities follow. For a complex one-form
$\xi=\xi_1e^1+\xi_2e^2+\xi_3e^3$,
\begin{equation}\label{eq:transverse-square}
 \ip{(B+\ii\mf)\xi}{\xi}
 =r\abs{\xi_1+\ii\xi_2}^2\ge0.
\end{equation}
All of these assertions are immediate when $F=0$.

We next verify the regularity at the zero set of $F$.
For Hermitian operators $H,K$ on the same finite-dimensional space,
\begin{equation}\label{eq:absolute-contraction}
 \norm{\abs H-\abs K}_{\HS}\le\norm{H-K}_{\HS}.
\end{equation}
Indeed, choose orthonormal eigenbases $v_i,w_j$ with eigenvalues
$\lambda_i,\mu_j$. The coefficients
$c_{ij}=\abs{v_i^*w_j}^2$ have all row and column sums equal to one.
Expanding the Hilbert--Schmidt norms gives
\begin{align*}
 \norm{\abs H-\abs K}_{\HS}^2
 &=\sum_{i,j}(\abs{\lambda_i}-\abs{\mu_j})^2c_{ij}\\
 &\le\sum_{i,j}(\lambda_i-\mu_j)^2c_{ij}
 =\norm{H-K}_{\HS}^2.
\end{align*}
This is the Hermitian case of the estimate in
\cite{Kittaneh1985}; its use for the curvature appears in
\cite[Section~3.1]{Badran2026}.

The absolute value is equivariant under orthogonal conjugation:
\begin{equation}\label{eq:equivariance}
 \abs{O^THO}=O^T\abs H\,O.
\end{equation}
Indeed, the right-hand side is nonnegative and its square is
$O^TH^2O$, so uniqueness of the nonnegative square root applies.
In a smooth local orthonormal trivialization, smoothness of $F$ and
\eqref{eq:absolute-contraction} imply that $B$ is locally Lipschitz.

We also give a covariant interpretation of the locally Lipschitz property. Let
$\gamma(0)=x$, $\dot\gamma(0)=X$, and let
$P_t:T_x\sphere\to T_{\gamma(t)}\sphere$ be parallel transport.
Pullback by $P_t$ identifies tensors along $\gamma$ with tensors at $x$.
By~\eqref{eq:equivariance},
\[
 P_t^*B_{\gamma(t)}=\abs{\ii P_t^*\mf_{\gamma(t)}}.
\]
Here pullback of an endomorphism means conjugation by the induced
parallel transport on cotangent spaces.
Consequently,
\[
 \norm{P_t^*B_{\gamma(t)}-B_x}_{\HS}
 \le\sqrt2\,\abs{P_t^*F_{\gamma(t)}-F_x}.
\]
Taking difference quotients at a differentiability point of $B$ yields
\[
 \norm{\nabla_XB}_{\HS}\le\sqrt2\,\abs{\nabla_XF}.
\]

Next we prove, in dimension three, the derivative norms are equal. On $\{F\ne0\}$,
write $r=\abs F$ and $\nu=*F/r$. As covariant tensors,
\[
 \abs\nu=1,\qquad F=r*\nu,\qquad
 B=r(g-\nu\otimes\nu).
\]
For any tangent vector $X$, metric compatibility gives
$\ip{\nabla_X\nu}\nu=0$, and hence
\[
 \nabla_XB=(Xr)(g-\nu\otimes\nu)
 -r\bigl((\nabla_X\nu)\otimes\nu
              +\nu\otimes\nabla_X\nu\bigr).
\]
The two terms on the right are orthogonal. Their squared norms give
\[
 \norm{\nabla_XB}_{\HS}^2
 =2(Xr)^2+2r^2\abs{\nabla_X\nu}^2.
\]
Since the Hodge star is parallel and isometric,
\[
 \nabla_XF=*((Xr)\nu+r\nabla_X\nu),\qquad
 \abs{\nabla_XF}^2=(Xr)^2+r^2\abs{\nabla_X\nu}^2.
\]
Thus the claimed equality holds on $\{F\ne0\}$.
On $\{F=0\}$, both $B$ and $F$ vanish. The locality property of weak
derivatives \cite[Section~9.8, item~4, p.~314]{Brezis2011}, applied
to their coefficients, gives $\nabla B=\nabla F=0$ almost everywhere
there; the connection terms vanish because the tensors themselves
vanish. Taking the metric trace over $X$ proves
\eqref{eq:absolute-derivative-equality}.
\end{proof}

\section{Stable Yang-Mills-Higgs Fields on $\sphere$ }\label{sec:direct}

\subsection{Two Integrated Bochner Identities}

We first record two Bochner identities. The endomorphisms $\mf$ and $B$
act on the one-form part of $Du$.

\begin{lemma}\label{lem:bochner}
At a smooth critical point on the unit round three-sphere,
\begin{equation}\label{eq:curvature-bochner}
\begin{split}
 2\kap\norm{\nabla F}_2^2+2\int\abs u^2\abs F^2\dd
 =-4\kap\norm F_2^2-2\int\ip{\ii\mf(Du)}{Du}\dd.
\end{split}
\end{equation}
Also,
\begin{equation}\label{eq:section-bochner}
\begin{split}
 \norm{\nabla^A Du}_2^2
 +\int\frac{3\abs u^2-1}{2\kap}\abs{Du}^2\dd
 =-2\norm{Du}_2^2-2\int\ip{\ii\mf(Du)}{Du}\dd.
\end{split}
\end{equation}
\end{lemma}

\begin{proof}
For tangent vectors $X,Y$, metric compatibility and $D^2=-\ii F$ give
\begin{align*}
 d\ip{Du}{\ii u}(X,Y)
 &=\ip{D^2u(X,Y)}{\ii u}
   +\ip{D_Yu}{\ii D_Xu}-\ip{D_Xu}{\ii D_Yu}\\
 &=-\abs u^2F(X,Y)+2\ip{\ii D_Xu}{D_Yu}.
\end{align*}
Contracting this identity with $F$ yields
\begin{equation}\label{eq:current-contraction}
 \ip{d\ip{Du}{\ii u}}F
 =-\abs u^2\abs F^2-\ip{\ii\mf(Du)}{Du}.
\end{equation}
The sign can be checked from~\eqref{eq:curvature-operator}:
in any orthonormal frame,
\[
 \ip{\ii\mf(Du)}{Du}
 =-\sum_{i,j}F(e_i,e_j)\ip{\ii D_{e_i}u}{D_{e_j}u}.
\]

For two-forms on the unit round three-sphere, the Weitzenb\"ock formula is
\[
 \Delta_HF=\nabla^*\nabla F+2F.
\]
Since $dF=0$ and $\kap\delta F=\ip{Du}{\ii u}$,
\[
 \kap\Delta_HF=d\ip{Du}{\ii u}.
\]
Pairing with $F$ and integrating by parts, we obtain
\[
 \kap\norm{\nabla F}_2^2+2\kap\norm F_2^2
 =-\int\abs u^2\abs F^2\dd
   -\int\ip{\ii\mf(Du)}{Du}\dd.
\]
Multiplication by two proves~\eqref{eq:curvature-bochner}.

For the second identity, let $d_D$ be the exterior covariant derivative
on complex-valued forms, and let $d_D^*$ be its formal adjoint.
On sections, $d_D=D$. The Weitzenb\"ock formula on complex-valued
one-forms is
\begin{equation}\label{eq:twisted-weitzenbock}
 (d_Dd_D^*+d_D^*d_D)\xi
 = (\nabla^A)^*\nabla^A\xi+2\xi+\ii\mf\xi.
\end{equation}
Here $2\xi$ is the Ricci term. The final term is the action of the
line-bundle curvature, with sign fixed by $D^2=-\ii F$ and
\eqref{eq:curvature-operator}.

Apply~\eqref{eq:twisted-weitzenbock} to $\xi=Du$.
Since
\[
 d_DDu=-\ii Fu,\qquad d_D^*Du=D^*Du,
\]
the product rule gives
\[
 d_D^*(Fu)=(\delta F)u+\mf(Du).
\]
It follows that
\begin{equation}\label{eq:commuted-section}
 (\nabla^A)^*\nabla^A Du
 =D(D^*Du)-2Du-\ii(\delta F)u-2\ii\mf(Du).
\end{equation}
Substitute the field equations~\eqref{eq:EL}. Differentiating the
potential explicitly,
\[
 D\left(\frac{1-\abs u^2}{2\kap}u\right)
 =\frac{1-\abs u^2}{2\kap}Du
   -\frac{1}{2\kap}u\,d\abs u^2.
\]
Pairing~\eqref{eq:commuted-section} with $Du$ and integrating yields
\begin{align*}
 \norm{\nabla^A Du}_2^2
 ={}&\int\left(\frac{1-\abs u^2}{2\kap}-2\right)
          \abs{Du}^2\dd\\
 &-\frac1\kap\int\left(
    \abs{\ip u{Du}}^2+\abs{\ip{Du}{\ii u}}^2\right)\dd
   -2\int\ip{\ii\mf(Du)}{Du}\dd.
\end{align*}
Finally, the real and imaginary parts of the Hermitian pairing give
\[
 \abs{\ip u{Du}}^2+\abs{\ip{Du}{\ii u}}^2
 =\abs u^2\abs{Du}^2.
\]
Rearrangement proves~\eqref{eq:section-bochner}.
\end{proof}

\subsection{The Conformal Variations}

Let $E_1,\ldots,E_4$ be the standard basis of $\R^4$.
For $x\in\sphere$, set
\[
 f_\alpha(x)=\ip{E_\alpha}x,\qquad
 X_\alpha=\nabla f_\alpha=E_\alpha-f_\alpha x.
\]
The sphere geometry gives
\begin{equation}\label{eq:sphere-derivative}
 \nabla_YX_\alpha=-f_\alpha Y
 \qquad\text{for every tangent vector }Y.
\end{equation}
The summation identities are
\begin{equation}\label{eq:sphere-frame}
 \sum_{\alpha=1}^4f_\alpha^2=1,\qquad
 \sum_{\alpha=1}^4f_\alpha X_\alpha=0,\qquad
 \sum_{\alpha=1}^4X_\alpha\otimes X_\alpha=g^{-1}.
\end{equation}
In particular, the last identity contracts any pair of one-forms
without choosing a local frame.We also recall that, with respect
to the $L^2$ inner product on vector fields,
\[
 \mathfrak{conf}(\sphere)
 =\mathfrak{isom}(\sphere)
  \oplus
  \operatorname{span}\{X_1,X_2,X_3,X_4\}.
\]
Indeed, $\operatorname{div}K=0$ for every Killing field $K$, so
\[
 \int_\sphere\ip K{X_\alpha}\dd
 =-\int_\sphere f_\alpha\operatorname{div}K\dd=0.
\]

Define the four variations as in \cite{Badran2026},
\begin{equation}\label{eq:W-definition}
 W_\alpha=(a_\alpha,w_\alpha)
 =\bigl(-B(X_\alpha,\cdot),\ii D_{X_\alpha}u\bigr),
 \qquad 1\le\alpha\le4.
\end{equation}
Since $B$ is locally Lipschitz and we work with a smooth representative
$(D,u)$, the variations $W_\alpha$ belong to $H^1$.


\noindent\textbf{Comparison with the standard conformal variations.}
The variations~\eqref{eq:W-definition} are closely related to the
standard conformal variations 
\[
 V_\alpha=(\iota_{X_\alpha}F,D_{X_\alpha}u).
\]
Indeed, on $\{F\ne0\}$, set $K=\mf/\abs F$. Since
\[
 B=\abs F\,P,\qquad K^2=-P,\qquad \mf=KB=BK,
\]
where $P$ is the orthogonal projection onto the curvature plane,
$\mf=KB$ is the polar decomposition of $\mf$. Indeed, our convention for $\mf$ gives
$\iota_XF=-\mf(X^\flat)$. We obtain
\[
 -K(\iota_XF)
 =K\mf(X^\flat)
 =K^2B(X^\flat)
 =-B(X,\cdot).
\]

If we define $\mathbb J_F(a,w)=(-Ka,\ii w)$, then $\mathbb J_F$ defines a complex structure on the subspace containing these variations and
\[
 W_\alpha=\mathbb J_FV_\alpha, V_\alpha=-\mathbb J_FW_\alpha.
\]
In this sense, $W_\alpha$ is obtained from
$V_\alpha$ by a quarter-turn in the curvature plane together with
multiplication by $\ii$ in the section component.  At points where $F=0$, both connection components vanish.


\begin{proposition}\label{prop:negative-trace}
At a smooth critical point, the variations~\eqref{eq:W-definition}
satisfy
\begin{equation}\label{eq:negative-trace}
\begin{split}
 \sum_{\alpha=1}^4Q(W_\alpha)
 ={}&-2\kap\norm F_2^2-\norm{Du}_2^2\\
 &-4\int\ip{(B+\ii\mf)(Du)}{Du}\dd\\
 &-\frac1\kap\norm{\kap\Div B+\frac12d\abs u^2}_2^2.
\end{split}
\end{equation}
Consequently,
\[
 \sum_{\alpha=1}^4Q(W_\alpha)
 \le-2\kap\norm F_2^2-\norm{Du}_2^2.
\]
\end{proposition}

\begin{proof}
All derivative identities below hold almost everywhere.
Integration by parts for $a_\alpha$ is justified by smooth
approximation in $H^1$.

\smallskip
\noindent\emph{Norms of the variations and their derivatives.}
For tangent vectors $Y,Z$, differentiation of~\eqref{eq:W-definition}
and~\eqref{eq:sphere-derivative} gives
\[
 (\nabla_Ya_\alpha)(Z)
 =-(\nabla_YB)(X_\alpha,Z)+f_\alpha B(Y,Z),
\]
and
\[
 D_Yw_\alpha
 =\ii(\nabla_Y^ADu)(X_\alpha)-\ii f_\alpha D_Yu.
\]
Using~\eqref{eq:sphere-frame} to square and sum, we obtain
\begin{equation}\label{eq:a-contractions}
\begin{aligned}
 \sum_\alpha\abs{a_\alpha}^2
 &=\norm B_{\HS}^2=2\abs F^2,\\
 \sum_\alpha\abs{\nabla a_\alpha}^2
 &=\norm{\nabla B}_{\HS}^2+\norm B_{\HS}^2.
\end{aligned}
\end{equation}
The same identities give
\begin{equation}\label{eq:w-contractions}
\begin{aligned}
 \sum_\alpha\abs{w_\alpha}^2&=\abs{Du}^2,\\
 \sum_\alpha\abs{Dw_\alpha}^2
 &=\abs{\nabla^ADu}^2+\abs{Du}^2.
\end{aligned}
\end{equation}
In both derivative calculations the cross terms vanish because
$\sum_\alpha f_\alpha X_\alpha=0$.

For a smooth real one-form $a$ on $\sphere$,
\[
 (d\delta+\delta d)a=\nabla^*\nabla a+2a.
\]
Taking its $L^2$ inner product with $a$ gives
\begin{equation}\label{eq:one-form-weitzenbock}
 \norm{da}_2^2+\norm{\delta a}_2^2
 =\norm{\nabla a}_2^2+2\norm a_2^2.
\end{equation}
This identity extends to $H^1$ one-forms by density.
Applying it to $a_\alpha$ and using~\eqref{eq:a-contractions},
\begin{equation}\label{eq:da-sum}
 \sum_\alpha\norm{da_\alpha}_2^2
 =\int\left(\norm{\nabla B}_{\HS}^2+6\abs F^2\right)\dd
   -\sum_\alpha\norm{\delta a_\alpha}_2^2.
\end{equation}

\smallskip
\noindent\emph{The mixed term.}
Since $a_\alpha^\sharp=-(B(X_\alpha,\cdot))^\sharp$, we have
\begin{equation}\label{eq:mixed-W}
\begin{split}
 \sum_\alpha\ip{\ii D_{a_\alpha^\sharp}u}{w_\alpha}
 &=-\sum_\alpha
 \ip{D_{(B(X_\alpha,\cdot))^\sharp}u}{D_{X_\alpha}u}\\
 &=-\ip{B(Du)}{Du}.
\end{split}
\end{equation}
The last equality is the contraction
$\sum_\alpha X_\alpha\otimes X_\alpha=g^{-1}$.

\smallskip
\noindent\emph{The divergence terms.}
Metric compatibility and~\eqref{eq:W-definition} imply
\begin{equation}\label{eq:section-pairing-W}
 \ip{\ii u}{w_\alpha}
 =\ip u{D_{X_\alpha}u}
 =\frac12d\abs u^2(X_\alpha).
\end{equation}
Taking the negative metric trace of $\nabla a_\alpha$ gives
\begin{equation}\label{eq:div-W}
 \delta a_\alpha=(\Div B)(X_\alpha)-f_\alpha\tr B.
\end{equation}
Here symmetry of $B$ identifies the contraction of $\nabla B$ with
the divergence in~\eqref{eq:div-convention}.
Combining~\eqref{eq:section-pairing-W} and~\eqref{eq:div-W},
\[
 \kap\delta a_\alpha+\ip{\ii u}{w_\alpha}
 =\left(\kap\Div B+\frac12d\abs u^2\right)(X_\alpha)
   -\kap f_\alpha\tr B.
\]
The three identities in~\eqref{eq:sphere-frame} now yield
\begin{equation}\label{eq:square-W}
\begin{split}
 \sum_\alpha\left(\kap\delta a_\alpha
                    +\ip{\ii u}{w_\alpha}\right)^2
 &=\abs{\kap\Div B+\frac12d\abs u^2}^2
      +\kap^2(\tr B)^2\\
 &=\abs{\kap\Div B+\frac12d\abs u^2}^2
      +4\kap^2\abs F^2.
\end{split}
\end{equation}
The divergence term from~\eqref{eq:da-sum} and the last two
pairing terms in~\eqref{eq:hessian-integrated} complete a square:
\begin{equation}\label{eq:completion}
\begin{split}
 &-\kap\abs{\delta a_\alpha}^2
 -2\ip{\ii u}{w_\alpha}\delta a_\alpha
 -\frac1\kap\bigl(\ip{\ii u}{w_\alpha}\bigr)^2\\
 &\hspace{3em}
 =-\frac1\kap\left(\kap\delta a_\alpha
                          +\ip{\ii u}{w_\alpha}\right)^2.
\end{split}
\end{equation}

\smallskip
\noindent\emph{Substitution into the Hessian.}
Insert~\eqref{eq:a-contractions}--\eqref{eq:completion}
into~\eqref{eq:hessian-integrated}. The curvature contribution
$6\kap\abs F^2$ in~\eqref{eq:da-sum} is reduced by
$4\kap\abs F^2$ from~\eqref{eq:square-W}. Grouping the two
Bochner expressions separately gives
\begin{equation}\label{eq:trace-before-bochner}
\begin{split}
 \sum_\alpha Q(W_\alpha)
 ={}&\int\left[\kap\norm{\nabla B}_{\HS}^2
                    +2\abs u^2\abs F^2\right]\dd\\
 &+\int\left[\abs{\nabla^ADu}^2
          +\frac{3\abs u^2-1}{2\kap}\abs{Du}^2\right]\dd\\
 &+2\kap\norm F_2^2+\norm{Du}_2^2
       -4\int\ip{B(Du)}{Du}\dd\\
 &-\frac1\kap\norm{\kap\Div B+\frac12d\abs u^2}_2^2.
\end{split}
\end{equation}
By~\eqref{eq:absolute-derivative-equality}, the first integral equals
the left-hand side of~\eqref{eq:curvature-bochner}.
The second is the left-hand side of~\eqref{eq:section-bochner}.
Substituting both identities gives
\begin{align*}
 \sum_\alpha Q(W_\alpha)
 ={}&-2\kap\norm F_2^2-\norm{Du}_2^2\\
 & -4\int\ip{(B+\ii \mf)(Du)}{Du}\dd\\
 &-\frac1\kap\norm{\kap\Div B+\frac12d\abs u^2}_2^2.
\end{align*}
This proves~\eqref{eq:negative-trace}.
The inequality follows from~\eqref{eq:absolute-positive}.
\end{proof}

\subsection{Classification of Stable Critical Points on $\sphere$}

\begin{proof}[Proof of Theorem~\ref{thm:main}]
Proposition~\ref{prop:regularity} reduces the proof to a smooth stable
critical point. Stability extends to the $H^1$ variations $W_\alpha$,
so~\eqref{eq:negative-trace} implies
\[
 0\le\sum_\alpha Q(W_\alpha)
 \le-2\kap\norm F_2^2-\norm{Du}_2^2.
\]
Therefore $F=Du=0$. Since $H^1_{\mathrm{dR}}(\sphere)=0$,
the closed one-form $A$ is exact, say $A=d\psi$.
The gauge transformation $e^{-\ii\psi}$ sends the pair to $(0,c)$,
where $c$ is constant because $Du=0$.
The scalar equation gives
\[
 (1-\abs c^2)c=0.
\]
The option $c=0$ is unstable: for a nonzero constant section $w$,
\begin{equation}\label{eq:zero-instability}
 Q_{(0,0)}(0,w)=-\frac1{2\kap}\norm w_2^2<0.
\end{equation}
Thus $\abs c=1$, and a constant gauge transformation sends $c$ to $1$.
Conversely, $E_\kap\ge0=E_\kap(0,1)$, so the vacuum is stable.
\end{proof}

\section{The lowest eigenvalue on the gauge quotient}\label{sec:spectrum}

Throughout this section, $(A,u)$ is a smooth critical point.
We first construct the orthogonal projection away from gauge directions,
and then use the Hessian trace to estimate the lowest eigenvalue.

\subsection{Orthogonal Projection Away From Gauge Directions}

Equip the space of variations with the weighted $L^2$ inner product
\[
 (Z,Z')_\kap
 =\int_{\sphere}
 \bigl(\kap\ip a{a'}+\ip w{w'}\bigr)\dd,
 \qquad Z=(a,w),\quad Z'=(a',w').
\]
The infinitesimal gauge direction associated with a smooth real function
$\chi$ is
\[
 \cT\chi=(d\chi,\ii u\chi).
\]
Integration by parts gives
\[
 (Z,\cT\chi)_\kap
 =\int_{\sphere}
 \chi\bigl(\kap\delta a+\ip{\ii u}w\bigr)\dd.
\]
Consequently, the orthogonal complement of the gauge directions is
characterized by
\[
 \cG Z=0,
 \qquad
 \cG(a,w)=\kap\delta a+\ip{\ii u}w.
\]
$\cG$ is the adjoint of $\cT$ with respect to the weighted inner product. And its composition with $\cT$ is 
\begin{equation}\label{eq:gauge-adjoint}
\begin{aligned}
 \cG\cT\chi
 &=(\kap\Delta+\abs u^2)\chi,
\end{aligned}
\end{equation}
where $\Delta=\delta d$ is the nonnegative scalar Laplacian.

\Needspace{12\baselineskip}
Let
\[
 Q(Z,Z')=\frac14\bigl(Q(Z+Z')-Q(Z-Z')\bigr)
\]
denote the symmetric bilinear form associated with $Q$.
For each fixed $t$, gauge invariance gives
\[
 E_\kap(A+ta+s\,d\chi,e^{\ii s\chi}(u+tw))
 =E_\kap(A+ta,u+tw).
\]
Differentiating first in $s$ and then in $t$ at $(s,t)=(0,0)$ yields
\begin{equation}\label{eq:gauge-radical}
\begin{split}
 0
 &=\left.\frac{d}{dt}\right|_{t=0}
   \delta E_{\kap,(A+ta,u+tw)}
   \bigl(d\chi,\ii(u+tw)\chi\bigr)\\
 &=Q(Z,\cT\chi)+\delta E_{\kap,(A,u)}(0,\ii w\chi)\\
 &=Q(Z,\cT\chi).
\end{split}
\end{equation}
Thus every infinitesimal gauge direction lies in the radical of the
Hessian $Q$. By density and continuity, these identities extend to
\[
 Z\in H^1(T^*\sphere;\R)\oplus H^1(\sphere;\C),
 \qquad \chi\in H^2(\sphere;\R).
\]

\begin{lemma}\label{lem:gauge-projection}
Every $Z\in H^1(T^*\sphere;\R)\oplus H^1(\sphere;\C)$ has an
orthogonal projection 
\[\Pi Z=Z-\cT\chi,\]
with
$\chi\in H^2(\sphere;\R)$, satisfying
\begin{equation}\label{eq:projection-properties}
\begin{aligned}
 \cG\Pi Z&=0,\\
 Q(\Pi Z)&=Q(Z),\\
 \norm Z_\kap^2
 &=\norm{\Pi Z}_\kap^2+\norm{\cT\chi}_\kap^2.
\end{aligned}
\end{equation}
\end{lemma}

\begin{proof}
We solve
\[
 (\kap\Delta+\abs u^2)\chi=\cG Z\in L^2(\sphere;\R).
\]
The identity
\[
 \int_{\sphere}\chi(\kap\Delta+\abs u^2)\chi\dd
 =\kap\norm{d\chi}_2^2+\norm{u\chi}_2^2
\]
shows that the kernel is trivial if $u\not\equiv0$.
If $u\equiv0$, the kernel consists of constants, and the right-hand side
is orthogonal to it because
\[
 \int_{\sphere}\cG Z\dd
 =\kap\int_{\sphere}\delta a\dd=0.
\]
The Fredholm alternative and elliptic regularity therefore give a
solution $\chi\in H^2$. When $u\equiv0$, we impose
$\int_{\sphere}\chi\dd=0$ to make the solution unique.

For $\Pi Z=Z-\cT\chi$, equation~\eqref{eq:gauge-adjoint} gives
\[
 \cG\Pi Z
 =\cG Z-(\kap\Delta+\abs u^2)\chi=0.
\]
It also gives
\[
 (\Pi Z,\cT\chi)_\kap
 =\int_{\sphere}\chi\cG\Pi Z\dd=0,
\]
which proves the asserted norm identity. Finally,
\eqref{eq:gauge-radical} implies
\[
 Q(\Pi Z)
 =Q(Z)-2Q(Z,\cT\chi)+Q(\cT\chi)
 =Q(Z).
 \qedhere
\]
\end{proof}

\subsection{Existence And Regularity Of Lowest Eigenvector}

Define
\begin{equation}\label{eq:mu-definition}
 \mu_1=\inf_{0\ne Z\in H^1,\ \cG Z=0}
             \frac{Q(Z)}{\norm Z_\kap^2}.
\end{equation}
On $\ker\cG$, we have $\ip{\ii u}w=-\kap\delta a$.
Substitution in~\eqref{eq:hessian-integrated} yields
\begin{equation}\label{eq:slice-hessian}
\begin{split}
 Q(a,w)=\int_{\sphere}\bigg[
 &\kap\bigl(\abs{da}^2+\abs{\delta a}^2\bigr)+\abs{Dw}^2
     +\abs u^2\abs a^2\\
 &+\frac{3\abs u^2-1}{2\kap}\abs w^2
     +4\ip{\ii D_{a^\sharp}u}w\bigg]\dd.
\end{split}
\end{equation}
The elementary estimates
\[
 \abs{Dw}^2\ge\frac12\abs{dw}^2-\abs A^2\abs w^2
\]
and
\[
 4\abs{\ip{\ii D_{a^\sharp}u}w}
 \le 4\abs{Du}\abs a\abs w
 \le 2\abs{Du}\bigl(\abs a^2+\abs w^2\bigr)
\]
control the lower-order terms. The integrated Weitzenb\"ock identity
\[
 \norm{da}_2^2+\norm{\delta a}_2^2
 =\norm{\nabla a}_2^2+2\norm a_2^2
\]
then gives the G{\aa}rding inequality
\begin{equation}\label{eq:garding}
 Q(Z)+C\norm Z_\kap^2\ge c\norm Z_{H^1}^2,
 \qquad Z\in H^1\cap\ker\cG,
\end{equation}
where $c>0$ and $C>0$ depend on the fixed smooth critical point and
$\kap$.

Choose a minimizing sequence with $\norm{Z_j}_\kap=1$.
The preceding estimate makes $(Z_j)$ bounded in $H^1$.
After passing to a subsequence,
\[
 Z_j\rightharpoonup Y \quad\text{in }H^1,
 \qquad
 Z_j\to Y \quad\text{in }L^2.
\]
The bounded linear map $\cG:H^1\to L^2$ preserves weak convergence,
so $\cG Y=0$. Strong $L^2$ convergence gives $\norm Y_\kap=1$.
The derivative terms in~\eqref{eq:slice-hessian} are weakly lower
semicontinuous. All remaining terms converge by strong $L^2$
convergence, since their coefficients are smooth and bounded. Hence
\[
 \mu_1\le Q(Y)\le\liminf_{j\to\infty}Q(Z_j)=\mu_1.
\]
Thus $Y$ attains the minimum.

The constrained Euler equation is
\[
 Q(Y,Z)=\mu_1(Y,Z)_\kap
 \qquad\text{for every }Z\in H^1\cap\ker\cG.
\]
For arbitrary $Z\in H^1$, the gauge projection and
\eqref{eq:gauge-radical} give
\[
 Q(Y,Z)=Q(Y,\Pi Z)=\mu_1(Y,\Pi Z)_\kap
 =\mu_1(Y,Z)_\kap.
\]
The weak eigenvalue equation therefore holds for arbitrary test pairs.

Write $Y=(a,w)$ and let $Z=(b,z)$ be any test pair.
Polarize the full Hessian~\eqref{eq:hessian-integrated} and use
$\ip{\ii u}w=-\kap\delta a$. The terms containing
$\ip{\ii u}z$ cancel, leaving
\[
\begin{split}
 Q(Y,Z)=\int_{\sphere}\bigg[
 &\kap\ip{da}{db}+\kap\delta a\,\delta b+\ip{Dw}{Dz}
    +\abs u^2\ip a b\\
 &+\frac{3\abs u^2-1}{2\kap}\ip w z
    +2\ip{\ii D_{a^\sharp}u}z
    +2\ip{\ii D_{b^\sharp}u}w\bigg]\dd.
\end{split}
\]
Integration by parts yields
\begin{equation}\label{eq:eigen-system}
\begin{cases}
 \displaystyle
 \Delta_Ha+\frac{\abs u^2}{\kap}a
       +\frac2\kap\ip{\ii Du}w=\mu_1a,\\[6pt]
 \displaystyle
 D^*Dw+\frac{3\abs u^2-1}{2\kap}w
       +2\ii D_{a^\sharp}u=\mu_1w.
\end{cases}
\end{equation}
Here $\ip{\ii Du}w$ is the real one-form whose value on a vector
$X$ is $\ip{\ii D_Xu}w$. The principal part of
\eqref{eq:eigen-system} is diagonal and elliptic, and its coefficients
are smooth. Elliptic regularity and iteration imply that $Y$ is smooth.

Under a gauge transformation, variations transform by
$(a,w)\mapsto(a,e^{\ii\chi}w)$. This preserves the Hessian, the
weighted norm, and the condition $\cG(a,w)=0$. Consequently,
$\mu_1$ is gauge invariant.

\subsection{The Eigenvalue Estimate}

\begin{proof}[Proof of Theorem~\ref{thm:gap}]
Suppose first that
\[
 2\kap\norm F_2^2+\norm{Du}_2^2>0.
\]
The contraction identities~\eqref{eq:a-contractions} and
\eqref{eq:w-contractions} give
\[
 \sum_\alpha\norm{W_\alpha}_\kap^2
 =2\kap\norm F_2^2+\norm{Du}_2^2.
\]
Projecting these variations onto $\ker\cG$ preserves their Hessians
and decreases their norms. Thus
\[
 \sum_\alpha\norm{\Pi W_\alpha}_\kap^2
 \le 2\kap\norm F_2^2+\norm{Du}_2^2,
\]
whereas the trace identity gives
\[
 \sum_\alpha Q(\Pi W_\alpha)
 =\sum_\alpha Q(W_\alpha)
 \le-2\kap\norm F_2^2-\norm{Du}_2^2<0.
\]
In particular, at least one projected variation is nonzero.
By the definition of $\mu_1$,
\[
\begin{aligned}
 \mu_1
 &\le
 \frac{\sum_\alpha Q(\Pi W_\alpha)}
      {\sum_\alpha\norm{\Pi W_\alpha}_\kap^2}\\
 &\le
 -\frac{2\kap\norm F_2^2+\norm{Du}_2^2}
       {\sum_\alpha\norm{\Pi W_\alpha}_\kap^2}
 \le-1.
\end{aligned}
\]

If $2\kap\norm F_2^2+\norm{Du}_2^2=0$, then $F=Du=0$.
Since $\sphere$ is simply connected, the pair is gauge equivalent to
$(0,c)$ for a constant $c\in\C$. The scalar equation gives
$(1-\abs c^2)c=0$. Therefore $(A,u)$ is gauge equivalent to either
$(0,1)$ or $(0,0)$. The former is the vacuum, so it remains to consider
$(0,0)$.

At $(0,0)$ the orthogonality condition is $\delta a=0$, and
\[
 Q(a,w)=\kap\norm{da}_2^2+\norm{dw}_2^2
             -\frac1{2\kap}\norm w_2^2.
\]
It follows that
\[
 Q(a,w)+\frac1{2\kap}\norm{(a,w)}_\kap^2
 =\kap\norm{da}_2^2+\norm{dw}_2^2+\frac12\norm a_2^2
 \ge0.
\]
Equality is attained when $a=0$ and $w$ is a nonzero constant.
Hence $\mu_1(0,0)=-1/(2\kap)$.
\end{proof}



\appendix
\section{Weak critical points}\label{app:regularity}

We prove that every weak critical point in
\eqref{eq:weak-class-intro} admits a smooth representative in Coulomb
gauge; compare
\cite[Appendix, Proposition~A.1 and Remark~A.3]{PigatiStern2021}.
Throughout the appendix, $H^k=W^{k,2}$.

\begin{lemma}\label{lem:weak-variation}
At $(A,u)\in\cC$, the first variation is continuous on $H^1$
variations, and~\eqref{eq:raw-hessian} defines a continuous quadratic
form there. Consequently, weak criticality and stability extend to
$H^1$ variations.
\end{lemma}

\begin{proof}
Since $A,u\in H^1$ and $u\in L^\infty$, we have
\[
 F=dA\in L^2,
 \qquad
 Du=du-\ii Au\in L^2.
\]
In dimension three, $H^1\hookrightarrow L^q$ for $2\le q\le6$.
In particular,
\[
\begin{aligned}
 \norm{Aw}_2&\le\norm A_6\norm w_3,\\
 \norm{au}_2&\le\norm u_\infty\norm a_2.
\end{aligned}
\]
The mixed Hessian term is controlled by
\[
 \left|\int_{\sphere}\ip{Du}{\ii aw}\dd\right|
 \le\norm{Du}_2\norm a_4\norm w_4.
\]
The potential coefficients are bounded because $u\in L^\infty$.
These estimates prove continuity of the first variation and the
polarized Hessian. Smooth approximation in $H^1$ then proves the
assertion.
\end{proof}

\begin{proposition}\label{prop:regularity}
Every weak critical point in $\cC$ is gauge equivalent, by
$e^{\ii\chi}$ with real $\chi\in H^2$, to a smooth critical point in
Coulomb gauge. This transformation preserves weak stability.
\end{proposition}

\begin{proof}
Since $\int_{\sphere}\delta A\dd=0$, the equation
\[
 \Delta\chi=-\delta A,
 \qquad
 \int_{\sphere}\chi\dd=0
\]
has a solution $\chi\in H^2(\sphere;\R)$.
Define the gauge transformation and the transformed pair by
\[
 \gamma=e^{\ii\chi},
 \qquad
 A'=A+d\chi,
 \qquad
 u'=\gamma u.
\]
The chain rule gives
\[
 d\gamma=\ii\gamma\,d\chi
\]
and
\[
 \nabla^2\gamma
 =\ii\gamma\nabla^2\chi
      -\gamma\,d\chi\otimes d\chi.
\]
Because $d\chi\in H^1\subset L^4\cap L^6$, it follows that
$\gamma\in H^2$. Moreover,
\[
 \norm{d\gamma\,w}_2
 \le\norm{d\gamma}_6\norm w_3.
\]
Thus multiplication by $\gamma$ and $\gamma^{-1}$ preserves $H^1$,
and $(A',u')$ belongs to $\cC$.

The transformed fields satisfy
\[
 \delta A'=0,
 \qquad
 F_{A'}=F_A,
 \qquad
 D_{A'}u'=\gamma D_Au.
\]
For every $H^1$ variation $(a,w)$, gauge invariance gives
\[
 E_\kap(A'+ta,u'+t\gamma w)
 =E_\kap(A+ta,u+tw).
\]
Differentiating twice yields
\begin{equation}\label{eq:weak-gauge-hessian}
 Q_{(A',u')}(a,\gamma w)=Q_{(A,u)}(a,w).
\end{equation}
Lemma~\ref{lem:weak-variation} also permits first differentiation
against $H^1$ tests. Hence $(A',u')$ is a weak critical point, and its
stability is equivalent to that of $(A,u)$.

In Coulomb gauge, the connection equation is
\begin{equation}\label{eq:connection-bootstrap}
 \kap\Delta_HA'
 =\ip{D_{A'}u'}{\ii u'}
 =\ip{du'}{\ii u'}-\abs{u'}^2A'.
\end{equation}
Its right-hand side belongs to $L^2$, since $u'\in L^\infty$ and
$D_{A'}u'\in L^2$. Elliptic regularity gives $A'\in H^2$.
The three-dimensional Sobolev embedding then gives
$A'\in L^\infty$.

Expanding the scalar equation using $\delta A'=0$, we obtain
\begin{equation}\label{eq:scalar-bootstrap}
 \Delta u'
 =\frac{1-\abs{u'}^2}{2\kap}u'
      -2\ii\,du'\bigl((A')^\sharp\bigr)-\abs{A'}^2u'.
\end{equation}
Every term on the right belongs to $L^2$: in particular,
\[
 \norm{du'((A')^\sharp)}_2
 \le\norm{A'}_\infty\norm{du'}_2
\]
and
\[
 \norm{\abs{A'}^2u'}_2
 \le\norm{A'}_\infty^2\norm{u'}_2.
\]
It follows that $u'\in H^2$ as well.

For each integer $m\ge2$, the three-dimensional Sobolev multiplication
estimates give
\[
 H^m\cdot H^m\subset H^m,
 \qquad
 H^m\cdot H^{m-1}\subset H^{m-1}.
\]
If $A',u'\in H^m$, both right-hand sides of
\eqref{eq:connection-bootstrap} and~\eqref{eq:scalar-bootstrap}
belong to $H^{m-1}$. Elliptic regularity therefore yields
$A',u'\in H^{m+1}$. Iteration and Sobolev embedding prove that both
fields are smooth.
\end{proof}

\noindent\textbf{Acknowledgement:} The author is supported by the Fundamental Research Funds for the Central Universities.

\noindent\textbf{AI Usage Statement:} The mathematical results and arguments in this paper are those of
the author. During the preparation of the manuscript, AI tools,
including Doubao and ChatGPT, were used to assist with language
editing and presentation of certain
calculations. All AI-assisted material was carefully verified and
revised by the author, who takes full responsibility for the
manuscript.


\begin{thebibliography}{99}

\bibitem{Badran2026}
M.~Badran,
\emph{Stable solutions in the abelian Higgs model},
preprint, arXiv:2609.11647v1 [math.AP] (2026).
\href{https://arxiv.org/abs/2609.11647v1}{arXiv:2609.11647v1}.

\bibitem{BGH2026}
M.~Badran, M.~A.~M.~Guaraco, and A.~Halavati,
\emph{Integral curvature estimates and geodesic limits for stable
abelian Higgs fields},
preprint, arXiv:2609.21893v1 [math.DG] (2026).

\bibitem{BethuelBrezisHelein1994}
F.~Bethuel, H.~Brezis, and F.~H\'elein,
\emph{Ginzburg--Landau Vortices},
Progress in Nonlinear Differential Equations and Their Applications,
vol.~13, Birkh\"auser, Boston, MA, 1994.
\href{https://doi.org/10.1007/978-1-4612-0287-5}
{doi:10.1007/978-1-4612-0287-5}.

\bibitem{Bogomolny1976}
E.~B.~Bogomol'nyi,
\emph{The stability of classical solutions},
Soviet J. Nuclear Phys. \textbf{24} (1976), no.~4, 449--454.

\bibitem{BourguignonLawson1981}
J.-P.~Bourguignon and H.~B.~Lawson, Jr.,
\emph{Stability and isolation phenomena for Yang--Mills fields},
Comm. Math. Phys. \textbf{79} (1981), no.~2, 189--230.
\href{https://doi.org/10.1007/BF01942061}
{doi:10.1007/BF01942061}.

\bibitem{BourguignonLawsonSimons1979}
J.-P.~Bourguignon, H.~B.~Lawson, and J.~Simons,
\emph{Stability and gap phenomena for Yang--Mills fields},
Proc. Natl. Acad. Sci. USA \textbf{76} (1979), no.~4, 1550--1553.
\href{https://doi.org/10.1073/pnas.76.4.1550}
{doi:10.1073/pnas.76.4.1550}.

\bibitem{Bradlow1990}
S.~B.~Bradlow,
\emph{Vortices in holomorphic line bundles over closed K\"ahler manifolds},
Comm. Math. Phys. \textbf{135} (1990), no.~1, 1--17.
\href{https://doi.org/10.1007/BF02097654}
{doi:10.1007/BF02097654}.

\bibitem{Brezis2011}
H.~Brezis,
\emph{Functional Analysis, Sobolev Spaces and Partial
Differential Equations},
Universitext, Springer, New York, 2011.

\bibitem{Cheng2020}
D.~R.~Cheng,
\emph{Instability of solutions to the Ginzburg--Landau equation on
$\mathbb S^n$ and $\mathbb {CP}^n$},
J. Funct. Anal. \textbf{279} (2020), no.~8, Paper No.~108669.
\href{https://doi.org/10.1016/j.jfa.2020.108669}
{doi:10.1016/j.jfa.2020.108669}.

\bibitem{Cheng2021}
D.~R.~Cheng,
\emph{Stable solutions to the abelian Yang--Mills--Higgs equations on
$\mathbb S^2$ and $\mathbb T^2$},
J. Geom. Anal. \textbf{31} (2021), no.~10, 9551--9572.
\href{https://doi.org/10.1007/s12220-021-00619-y}
{doi:10.1007/s12220-021-00619-y}.

\bibitem{DHPExcess2026}
G.~De~Philippis, A.~Halavati, and A.~Pigati,
\emph{Decay of excess for the abelian Higgs model},
J. Eur. Math. Soc. (2026), published online.
\href{https://doi.org/10.4171/JEMS/1819}
{doi:10.4171/JEMS/1819}.

\bibitem{DePhilippisPigati2024}
G.~De~Philippis and A.~Pigati,
\emph{Non-degenerate minimal submanifolds as energy concentration sets:
a variational approach},
Comm. Pure Appl. Math. \textbf{77} (2024), no.~8, 3581--3627.
\href{https://doi.org/10.1002/cpa.22193}
{doi:10.1002/cpa.22193}.

\bibitem{GarciaPrada1994}
O.~Garc\'ia-Prada,
\emph{A direct existence proof for the vortex equations over a compact
Riemann surface},
Bull. London Math. Soc. \textbf{26} (1994), no.~1, 88--96.
\href{https://doi.org/10.1112/blms/26.1.88}
{doi:10.1112/blms/26.1.88}.

\bibitem{HanJinWen2023}
X.~Han, X.~Jin, and Y.~Wen,
\emph{Stability and energy identity for Yang--Mills--Higgs pairs},
J. Math. Phys. \textbf{64} (2023), no.~2, Paper No.~021511.
\href{https://doi.org/10.1063/5.0130905}{doi:10.1063/5.0130905}.

\bibitem{JaffeTaubes1980}
A.~Jaffe and C.~Taubes,
\emph{Vortices and Monopoles: Structure of Static Gauge Theories},
Progress in Physics, vol.~2, Birkh\"auser, Boston, MA, 1980.

\bibitem{Kittaneh1985}
F.~Kittaneh,
\emph{On Lipschitz functions of normal operators},
Proc. Amer. Math. Soc. \textbf{94} (1985), no.~3, 416--418.
\href{https://doi.org/10.1090/S0002-9939-1985-0787884-4}
{doi:10.1090/S0002-9939-1985-0787884-4}.

\bibitem{LawsonSimons1973}
H.~B.~Lawson, Jr. and J.~Simons,
\emph{On stable currents and their application to global problems
in real and complex geometry},
Ann. of Math. (2) \textbf{98} (1973), no.~3, 427--450.
\href{https://doi.org/10.2307/1970913}
{doi:10.2307/1970913}.

\bibitem{MarxKuo2025}
J.~Marx-Kuo,
\emph{Second inner variations, stress-energy tensors, and index
of limiting varifolds},
J. Geom. Anal. \textbf{35} (2025), no.~10, Paper No.~290.
\href{https://doi.org/10.1007/s12220-025-02123-z}
{doi:10.1007/s12220-025-02123-z}.

\bibitem{PPSGamma2024}
D.~Parise, A.~Pigati, and D.~Stern,
\emph{Convergence of the self-dual $U(1)$-Yang--Mills--Higgs energies
to the $(n-2)$-area functional},
Comm. Pure Appl. Math. \textbf{77} (2024), no.~1, 670--730.
\href{https://doi.org/10.1002/cpa.22150}
{doi:10.1002/cpa.22150}.

\bibitem{PPSFlow2024}
D.~Parise, A.~Pigati, and D.~Stern,
\emph{The parabolic $U(1)$-Higgs equations and codimension-two
mean curvature flows},
Geom. Funct. Anal. \textbf{34} (2024), 1171--1225.
\href{https://doi.org/10.1007/s00039-024-00684-9}
{doi:10.1007/s00039-024-00684-9}.

\bibitem{PigatiStern2021}
A.~Pigati and D.~Stern,
\emph{Minimal submanifolds from the abelian Higgs model},
Invent. Math. \textbf{223} (2021), no.~3, 1027--1095.
\href{https://doi.org/10.1007/s00222-020-01000-6}
{doi:10.1007/s00222-020-01000-6}.

\bibitem{Simons1968}
J.~Simons,
\emph{Minimal varieties in Riemannian manifolds},
Ann. of Math. (2) \textbf{88} (1968), no.~1, 62--105.
\href{https://doi.org/10.2307/1970556}
{doi:10.2307/1970556}.

\bibitem{Stern2010}
M.~Stern,
\emph{Geometry of minimal energy Yang--Mills connections},
J. Differential Geom. \textbf{86} (2010), no.~1, 163--188.
\href{https://doi.org/10.4310/jdg/1299766686}
{doi:10.4310/jdg/1299766686}.

\bibitem{Taubes1980}
C.~H.~Taubes,
\emph{Arbitrary $N$-vortex solutions to the first order
Ginzburg--Landau equations},
Comm. Math. Phys. \textbf{72} (1980), no.~3, 277--292.
\href{https://doi.org/10.1007/BF01197552}
{doi:10.1007/BF01197552}.

\end{thebibliography}
\end{document}